\documentclass[12pt,runningheads,a4paper]{amsart}
\usepackage[utf8]{inputenc}
\usepackage{amsfonts}
\usepackage{amsthm}
\usepackage{amscd}
\usepackage{t1enc}
\usepackage[mathscr]{eucal}
\usepackage{indentfirst}
\usepackage{graphics}
\usepackage{amssymb}
\usepackage{amsmath}
\usepackage{verbatim}
\usepackage[margin=2.9cm]{geometry}
\usepackage{epstopdf} 
\usepackage[utf8]{inputenc}
\usepackage[english]{babel}
\usepackage{cite}
\usepackage{xcolor}

\usepackage{caption}
\usepackage{subcaption}
\usepackage{algorithm}
\usepackage[noend]{algorithmic}

\usepackage[pdftex]{graphicx}

\newcommand{\zz}{\ensuremath{\mathbb{Z}}}
\newcommand{\qq}{\ensuremath{\mathbb{Q}}}

\newcommand{\cc}{\ensuremath{\mathbb{C}}}

\graphicspath{ {./Desktop/ } }

\theoremstyle{plain}
\newtheorem{Th}{Theorem}[section]
\newtheorem{Lemma}[Th]{Lemma}
\newtheorem{Cor}[Th]{Corollary}
\newtheorem{Prop}[Th]{Proposition}
\theoremstyle{definition}

\newtheorem{Rem}[Th]{Remark}
\newtheorem{?}[Th]{Problem}

\newtheorem{Alg}[Th]{Algorithm}

\begin{document}
	
	\title{Divisors in Residue Classes Revisited}

	\author{Jonathon Hales}
	
	\address{ }
	
	\email{ hales.41@osu.edu}
	
	\maketitle
	\begin{abstract}
		In 1984, H. W. Lenstra described an algorithm finding divisors of $N$ congruent to $r \mod S$. When $S^3 > N$, this algorithm runs in polynomial time and hence factors $N$ in time $N^{1/3+o(1)}$. Lenstra's algorithm relies on a sign change in a constructed sequence and so cannot be adapted directly to larger euclidean number rings. We present a new method that generalizes to a larger class of euclidean rings and the polynomial ring $\zz[x]$. The algorithm is implemented and timed confirming its polynomial run time.
	\end{abstract}
	\section{Introduction}\label{intro}
	In 1984, H. W. Lenstra in \cite{DivsInRes} described an algorithm that finds all the divisors of an integer $N$ congruent to $r \mod S$ with $(N,S) = (S,r) = 1$. This can be rephrased as finding all solutions in nonnegative integers $x,y$ to the equation \begin{equation*}
	(Sx+r)(Sy+r') = N.
	\end{equation*}
	When $S^3 > N$, his algorithm runs in polynomial time. 
	
	Lenstra's method relies critically on a sign change in a sequence of terms of the form $a_ix +b_iy$ for some integers $a_i,b_i$ and so is not immediately adaptable to larger number rings in which there is no natural ordering of the terms. We show how this can be done for imaginary quadratic Euclidean domains in Section \ref{ImagQuadSec} and prove the following.
	\begin{Th}\label{ImagQuadTh}
		Let $K = \qq(\sqrt{d})$ for $d = -1,-2,-3,-7,$ or $-11$, and $\mathcal{O}_K$ its ring of integers. If $N,S,r \in \mathcal{O}_K$ such that $\gcd(N,S) = (S,r) = 1$. Then there exists an algorithm which outputs the divisors $d$ of $N$ of the form $d = Sx+r$ for some $x \in \mathcal{O}_K$. When $|S|^3 > |N|$ the algorithm runs 
		$
		O\left( \log(|N|)\right)
		$
		Elementary operations in $\mathcal{O}_K$.
	\end{Th}
	\begin{Rem}
		Elementary operations refers to additions, multiplications, subtractions, integer divisions, and square roots extractions.
		The values of $d$ referenced in Theorem \ref{ImagQuadTh} constitute a complete list of negative integers such that $\mathcal{O}_K$ is a Euclidean domain with respect to the complex absolute value. 
	\end{Rem}
	We then show in Section \ref{PolySec} how a slight modification of the algorithm referenced in Theorem $\ref{ImagQuadTh}$ can be adapted to finding divisors in $\zz[x]$ and prove the following. 
	\begin{Th}\label{PolyTh}
		If $N,S,r \in \zz[x]$ are monic polynomials with $\gcd(N,S) = \gcd(S,r) = 1$ and $3\cdot\deg(S)\geq \deg(N)$, then there exists an algorithm which outputs the divisors $d$ of $N$ of the form $d = Sf+r$ for some $f \in \zz[x]$. It runs in $O\left(\deg(N)\right)$ elementary operations in $\zz[x]$.
	\end{Th}
	\begin{Rem}
		Elementary operations again refers to additions, multiplications, subtractions, integer divisions, and square roots extractions. 
		If the polynomial $N$ is not monic, a version of Theorem \ref{PolyTh} still exists but the factorization of the leading term of $N$ must be known. 
		Otherwise, for some integer $n = pq$ we could recover $p$ and $q$ in polynomial time by thinking of $n$ as a degree zero polynomial and searching for all of the divisors congruent to $0 \mod 1$. 
	\end{Rem}
	In Section \ref{BoundsSec} we discuss Lenstra's combinatorial model for counting the number of divisors an integer can have in a given residue class, and comment on how this model can be adapted to the case of larger number rings.
	We then conclude and highlight some areas of future work in Section \ref{conc}.
	\section{The Algorithm in $\qq\left(\sqrt{d}\right)$}\label{ImagQuadSec}
	The majority of this section is devoted to carefully detailing the modified Lenstra algorithm in $\zz[i]$. The ideas behind this modification generalize easily to the other cases in Theorem \ref{ImagQuadTh} and we prove the appropriate lemmas. We will begin with a detailed description of Lenstra's original method and conclude with a brief discussion on the difficulties of adapting this algorithm to other number rings which still admit a Euclidean algorithm, but that our theorem does not apply to (like $\zz[\sqrt{2}]$ for example).
	\subsection{Lenstra's Method}
	The key idea behind Lenstra's method is to combine the equation
	\begin{equation}\label{ProductForm}
	(Sx+r)(Sy+r') = N
	\end{equation}
	with a sequence of congruence conditions of the form
	\[
	a_ix + b_iy \equiv c_i \mod S.
	\]
	If one can show that the absolute value of $|a_ix + b_iy|$ is small relative to $S$, then the congruence can be converted to a finite number of \emph{equalities}. 
	The resulting equalities can be combined with Equation $\left(\ref{ProductForm}\right)$ to solve for $x$ and $y$. 
	
	Specifically, assuming $x,y$ satisfy Equation $\left(\ref{ProductForm}\right)$, then expanding and looking modulo $S$ gives
	\[
	a_0x +b_0y \equiv c_0 \mod S \quad \quad \text{and}\quad \quad a_1x+b_1y \equiv c_1 \mod S
	\]
	where $a_0 = S$,$a_1 = r'$,$b_0 = 0$,$b_1 = r$ and $c_0 = 0$,$c_1 \equiv (N-rr')/S$. 
	From these initial conditions we can chain these congruences into a larger sequence by observing the following. For any integer $q$, we have 
	\begin{align*}
	(a_0-qa_{1})x + (b_0-qb_{1})y &\equiv c_0-qc_1 \mod S. 
	\end{align*}
	Choosing $q$ via the division algorithm gives a value $a_2 := a_0 -q a_{1}$ so that $a_2 < a_1$. This process is iterated, defining via the division algorithm integers $q_i$ so that 
	\[
	a_i = a_{i-2} - q_{i}a_{i-1}
	\]
	with $0\leq a_i < a_{i-1}$ for $i$ even and $0 < a_i \leq a_{i-1}$ for $i$ odd. This is to ensure that the sequence terminates (i.e. $a_i = 0$) at an even index $i$. 
	As before, to preserve the relation $a_ix +b_iy \equiv c_i \mod S$, the values $b_i,c_i$ are defined as 
	\[
	b_i = b_{i-2}-q_ib_{i-1}\quad \quad \quad c_i \equiv c_{i-2}-q_ic_{i-1}\mod S.
	\]
	
	By induction the values of $b_i$ are nonnegative for $i$ odd and nonpositive for $i$ even. Lenstra was able to use this sign changing to show that there exists an even index $i$ so that $a_ix+b_iy \geq 0$ and $a_{i+2}x+b_{i+2}y\leq 0$, and that one of the terms $a_{j}x+b_{i}y$ with $j = i,i+1$ or $i+2$ must satisfy
	\[
	-S < a_jx + b_jy < S
	\]
	for $j$ even and 
	\[
	2a_jb_j < a_jx+b_jy < a_jb_j + N/S^2
	\]
	for $j$ odd. 
	If $S^3 > N$ then in either case the additional information that $a_jx+b_jy \equiv c_j$ reduces the possible values of $a_jx+b_jy$ to just $1$ or $2$ possible values. 
	
	Finally, $x$ and $y$ are recovered by solving the resulting quadratic equation $(Sx+r)(Sy+r') = N$ given the value of $xa_j+yb_j$ where $j$ is the critical index where the sign change occurs. Since $j$ is unknown a priori (since $x$ and $y$ are unknown), each of the $O(\log(S))$ terms is tested. In total this gives a $O(\log(N)^{2+\epsilon})$ for any $\epsilon > 0$ for determining all positive divisors of $N$ congruent to $r \mod S$.
	\subsection{The Algorithm in $\zz[i]$}
	There are two main difficulties in adapting Lenstra's algorithm to larger number rings. First, Lenstra uses the positivity of $x,y$ in several key places to obtain his even and odd term inequalities. Second, there is no clear generalization of the sign change in Lenstra's algorithm to help us restrict the terms $a_ix +b_i y$ to a small set of known residue systems. 
	
	Our contribution is a new method that does not rely on a sign change argument to bound $|a_ix+b_iy|$ by an explicit constant times $|S|$. This allows us to limit our search for solutions to equation $\left(\ref{ProductForm}\right)$ to just $O(\log(|S|))$ candidates we check individually.
	Also of note, since this method bounds the terms $a_ix+b_iy$ in absolute value, all integer solutions to equation $\left(\ref{ProductForm}\right)$ are found, not just those restricted to a particular quadrant (or sign in the rational integer case).
	
	We begin with a brief reminder of how Euclidean division works in $\zz[i]$. 
	\begin{Lemma}\label{EuclidDivZi}
		Let $\alpha = a_0+a_1i, \beta = b_0+b_1i \in \zz[i]$ with $\beta \neq 0$. Then there exists $q,r \in \zz[i]$ such that $\alpha = q\beta +r$ with $|r| \leq \frac{1}{\sqrt{2}}|\beta|$.
	\end{Lemma}
	\begin{proof}
		We write 
		\[
		\frac{\alpha}{\beta} = \frac{a_0+a_1i}{b_0+b_1i} = \frac{(a_0+a_1i)(b_0-b_1i)}{b_0^2+b_1^2} = \frac{a_0b_0+a_1b_1}{b_0^2+b_1^2} + \frac{a_1b_0-a_0b_1}{b_0^2+b_1^2}i.
		\]
		Pick the integers $c,d$ closest to $\frac{a_0b_0+a_1b_1}{b_0^2+b_1^2}$ and $\frac{a_1b_0-a_0b_1}{b_0^2+b_1^2}$ respectively (ties broken by rounding down) and set $q = c+di$. Then 
		\[
		\left| \frac{\alpha}{\beta}- q\right|^2 \leq \frac{1}{2^2}+\frac{1}{2^2} = \frac{1}{2}
		\]
		so that 
		\[
		\left|\alpha - \beta q\right| \leq \frac{1}{\sqrt{2}}\left|\beta\right|. \qedhere
		\]
	\end{proof}
	We can now present our modification of Lenstra's algorithm.
	\begin{Alg}\label{ZiAlg}
		\mbox{}
		
		\noindent	\textit{Inputs:}
		\begin{itemize}
			\item[-] Gaussian integers $S,N,r$ with $|S|^3 > |N|$, and $\gcd(N,S) = \gcd(S,r) = 1$.
		\end{itemize}
		\textit{Outputs:}
		\begin{itemize}
			\item[-] A complete list of Gaussian integers $d$ such that $d$ divides $N$, and $d \equiv r \mod S$.
		\end{itemize}
		\begin{algorithmic}[1]
			\STATE Set $r = r \mod S$, and initialize $a_0 = S, a_1 = r'r^{-1} \mod S$, $b_0 = 0, b_1 = 1$, $c_0 = 0, c_1 = \frac{N-rr'}{S}r^{-1} \mod S$. Finally, set $r' = Nr^{-1} \mod S$ and set $t = 1$.
			\STATE If $r | N$, add $r$ to the list of outputs. If $r'|N$, add $N/r'$ to the list of outputs.
			\WHILE{$|a_t| > 0$}
			\STATE Calculate $q_t,r_t$ via Lemma \ref{EuclidDivZi} so that $a_{t-1} = a_{t}q_{t} + r_t$ where $|r_t| \leq |a_{t}/\sqrt{2}|$.
			\STATE Set $t = t+1$, and $a_t = r_{t-1}$. In addition, set $b_t = b_{t-2} - q_{t-1} b_{t-1}$ and set $c_t = c_{t-2}-q_{t-1}c_{t-1} \mod S$.
			\ENDWHILE 
			\FOR{$i=1,\dots, t$}
			\STATE Calculate the Gaussian integers $[\gamma_1,\dots,\gamma_m]$ such that $\gamma_k \equiv c_i \mod S$ and $|\gamma_k| < 12|S|$ holds for all $k$.
			\FOR{$j = 1,\dots,m$}
			\STATE Solve the system $a_ix+b_iy = \gamma_j$ and $(Sx+r)(Sy+r') = N$ for $x,y$. If $x,y$ are elements of $\zz[i]$ add $d = Sx+r$ to the list of outputs.
			\ENDFOR
			\ENDFOR
			\STATE Return the list of outputs.
		\end{algorithmic}
	\end{Alg}
	\begin{Prop}\label{ZiAlgTh}
		Algorithm $\left(\ref{ZiAlg}\right)$ is correct, and runs in $O(\log|N|)$ elementary operations on Gaussian integers of absolute value $O(|N|)$. This proves Theorem $\left(\ref{ImagQuadTh}\right)$ when $d = -1$.
	\end{Prop}
	\begin{proof}
		We first show correctness and begin with some properties of the sequences $a_k,b_k$, and possible solutions $x,y$. Our goal is to bound $|a_kx+b_ky|$ by $12|S|$ for some $k$ by arguing that there must be a term $k$ such that $|a_kx|$ and $|b_ky|$ are roughly the same size. Lemma \ref{ABLemma} shows us how to control the terms $a_k,b_k$, while Lemma \ref{XYLemma} will allow us to control $|x|,|y|$ and $|xy|$.
		\begin{Lemma}\label{ABLemma}
			For any $k\geq 0$ we have $$a_kb_{k+1}-a_{k+1}b_k =(-1)^k S$$ and $$|a_{k}b_{k+1}| \leq \left(2-\frac{1}{2^k}\right)|S|.$$
		\end{Lemma}
		\begin{proof}
			We argue inductively. By construction, $a_0b_1 - a_1b_0 = S - 0 = S$, and for any $k\geq 0$ we have that 
			\begin{align*}
			a_{k+1}b_{k+2} - a_{k+2}b_{k+1} & = a_{k+1}(b_k-q_{k+1}b_k) -(a_{k}-q_{k+1}a_{k+1})b_{k+1}\\
			&=a_{k+1}b_k - a_{k}b_{k+1}\\
			&= (-1)(-1)^{k}S = (-1)^{k+1}S
			\end{align*} where the penultimate equality follows by the inductive hypothesis. This proves the first claim.
			
			Next, for $k = 0$ we have $|a_kb_{k+1}| = |S| = \left(2-\frac{1}{2^k}\right)\cdot |S|$. For $k = 1$ we have that $a_1b_2 - a_2b_1 = -S$. Since $b_1 = 1$ and $|a_2| \leq |a_1|/\sqrt{2}\leq |S|/2$, we have 
			\[
			|a_1b_2| - \frac{|S|}{2}\leq |a_1b_2|-|a_2b_1| \leq |a_1b_2 - a_2b_1| = |S|
			\]
			so that $|a_1b_2| \leq \frac{3}{2}|S| = \left(2-\frac{1}{2^k}\right)|S|$.
			Finally, for any $k \geq 1$ we have that 
			\begin{align*}
			|S| = |a_{k}b_{k+1}-a_{k+1}b_{k}| &\geq |a_{k}b_{k+1}|-|a_{k+1}b_k|\\
			&\geq |a_{k}b_{k+1}| - \frac{1}{2}|a_{k-1}b_{k}| \geq |a_{k}b_{k+1}| - \left(1-\frac{1}{2^{k}}\right)|S|
			\end{align*}
			so that 
			\[
			\left( 2-\frac{1}{2^{k}}\right)|S| \geq |a_{k}b_{k+1}|. \qedhere
			\]
		\end{proof}
		\begin{Lemma}\label{XYLemma}
			For nonzero $x,y,r,r',S,N \in \zz[i]$, if $(Sx+r)(Sy+r') = N$ with $|r|,|r'| < |S|/\sqrt{2}$ and $|S|^3 > |N|$, then $|x|,|y| \leq 5|S|$ and  $|xy| < 9 |S|$.
		\end{Lemma}
		\begin{proof}
			We start with the equation
			\[
			Sx+r = \frac{N}{Sy+r'}
			\]
			so that 
			\[
			x = \frac{N}{S(Sy+r')} -\frac{r}{S}.
			\]
			The triangle inequalities then give 
			\begin{align*}
			|x| \leq \frac{|N|}{|S| (|Sy|-|r'|)} + \frac{|r|}{|S|} &\leq \frac{|N|}{|S| (|S|-|S|/\sqrt{2})} + \frac{|r|}{|S|} \\
			&\leq \frac{1}{1-1/\sqrt{2}}\frac{|N|}{|S^2|}+\frac{1}{\sqrt{2}}\\
			&\leq \left(\frac{\sqrt{2}}{\sqrt{2}-1}+\frac{1}{\sqrt{2}}\right)|S|\\
			&< 5|S|.
			\end{align*}
			The bound for $|y|$ is analogous.
			Next, since
			\[
			N = (Sx+r)(Sy+r') = S^2xy + S(xr'+yr)+rr',
			\]
			the triangle inequality gives
			\[
			|xy| = \left| \frac{N-S(xr'+yr)-rr'}{S^2}\right| \leq \left|\frac{N}{S^2}\right| + \left| \frac{xr'+yr}{S}\right| + \left|\frac{rr'}{S^2}\right|.
			\]
			Using $|r|,|r'| < |S|/\sqrt{2}$ gives that 
			\[
			|xy| \leq |S| + \frac{1}{\sqrt{2}}|x+y| + \frac{1}{2}.
			\]
			Finally, since each $|x|,|y| < 5|S|$, we have that 
			\[
			|xy| < |S| + \frac{10}{\sqrt{2}}|S| +\frac{1}{2}|S| < 9|S|. \qedhere
			\]
		\end{proof}
		\noindent Combining the above lemmas allows us to prove the following.
		\begin{Lemma}\label{SmallTerm}
			Under the hypotheses of Lemma $\ref{XYLemma}$, there exists $k$ with $$|a_k x +b_ky| < 12|S|.$$
		\end{Lemma}
		\begin{proof}
			Since the sequence $a_i$  begins at $a_0 = S$ and terminates at $a_t = 0$, and is strictly decreasing in absolute value, by Lemma \ref{XYLemma}, there exists $k\geq 1$ so that $5|a_{k}|\leq |y| < 5|a_{k-1}|$. Then 
			\[
			|a_{k}x+b_{k}y| \leq |a_kx| + |b_ky| \leq \frac{1}{5}|xy| + 5|b_ka_{k-1}| \leq  \frac{9}{5}|S|+10|S| < 12|S|. \qedhere
			\]
		\end{proof}
		We now prove correctness of Algorithm \ref{ZiAlg}. If $x,y$ are solutions to the equation $(Sx+r)(Sy+r') = N$, we must show that the algorithm discovers it. We first note that for all $k$ that $a_kx + b_ky \equiv c_k \mod S$. This is trivially true for $k = 0$, and for $k = 1$ we have from equation $\left(\ref{ProductForm}\right)$ that 
		\[
		S(xr'+yr) \equiv N-rr' \mod S^2
		\]
		so that 
		\[
		xr'r^{-1} + y \equiv \frac{N-rr'}{S}r^{-1} \mod S
		\]
		and thus 
		\[
		a_1x+b_1 y \equiv c_1 \mod S.
		\]
		Now for $k\geq 1$ we have that 
		\begin{align*}
		a_{k+1}x+b_{k+1} &= (a_{k-1}-q_ka_k)x + (b_{k-1}-q_kb_k)y\\
		&= a_{k-1}x + b_{k-1}y -q_k(a_kx+b_ky)\\
		&\equiv c_{k-1}-q_kc_{k} \equiv c_{k+1} \mod S.
		\end{align*}
		By Lemma $\left(\ref{SmallTerm}\right)$ one of the terms $a_kx+b_ky$ is smaller in absolute value than $12|S|$. Thus for some $k$, $a_kx+b_ky$ is one of the terms $\gamma_*$ discovered in line $7$. This means when $\gamma_j = \gamma_*$ in step $9$, both $x$ and $y$ are discovered. 
		
		We now consider the run time. Both steps $1$ and $2$ are trivial. Each iteration of the loop entered in step $3$ requires one (Gaussian) integer division, and one integer addition. Since $|a_{k+1}| \leq |a_k|/\sqrt{2}$, this loop is entered at most $O(\log|S|)$, and similarly $t = O(\log(|S|))$.
		
		The loop entered in line $6$ requires at most $O(mt)$ additions, multiplications and square root extractions of Gaussian integers whose real and imaginary parts are bounded by $O(|S|)$. Since $t = O(\log(|S|))$, to prove the theorem we must show that $m = O(1)$. Partitioning the complex plane into complete residue systems mod $S$ consisting of squares of area $|S|^2$, gives at most a finite number (independent of $|S|$) of squares that can fit inside a circle of radius $12|S|$. A crude upper bound would be for example $24^2$, the area of the square whose side length is the diameter of the circle. This completes the proof of Proposition $\ref{ZiAlgTh}$.
	\end{proof}
	We implemented Algorithm \ref{ZiAlg} in both \verb!Python! and \verb!C++!. For the \verb!C++! implementation we used the GNU Multiple Precision Arithmetic Library (GMP). GMP is a free library for operations on signed integers of arbitrary size. For each integer $10\leq k \leq 50$ we generated $2$ random signed integers $u_i$ satisfying $10^k\leq |u_i| < 10^{k+1}$ and used these values of $u_i$ to generate the real and imaginary parts of a Gaussian integer $N$. In a similar manor we picked a Gaussian integer $S$ whose (unsigned) real and imaginary parts belong to the interval $[10^{k/3},10^{(k+1)/3})$ and a Gaussian integer $r$ satisfying $|r| \leq |S|/\sqrt{2}$. If $\gcd(N,S) \neq 1$ or $\gcd(S,r) \neq 1$ or if $|S|^3 < |N|$ we reject this sample and try again using this same process. This was done until $100$ samples were selected for each $k$. The algorithm was then executed and timed on all $100$ inputs. The average of those run times is plotted for each $k$ in Figure \ref{ZiAlgTimes}, and highlight the polynomial runtime of Algorithm $\ref{ZiAlg}$.
	\begin{figure}[ht]
		\centering
		\includegraphics[scale=.6]{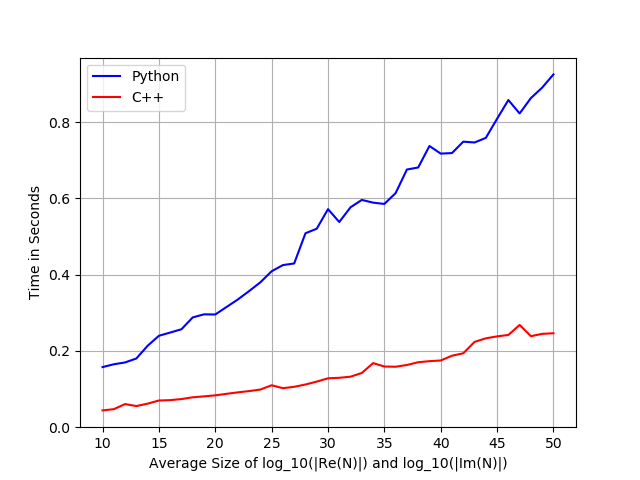}
		\caption{Timings of Algorithm \ref{ZiAlg} with large real and imaginary parts}
		\label{ZiAlgTimes}
	\end{figure}
	\subsection{Other imaginary Quadratic Euclidean Domains}
	In this section $\mathcal{O}_K$ will refer to the ring of integers for $K:= \qq(\sqrt{d})$ for $d=-1,-2, -3,-7,-11$. We will be able to form a version of Algorithm \ref{ZiAlg} if we can prove the appropriate versions of Lemmas \ref{ABLemma}, \ref{XYLemma}, and \ref{SmallTerm} which we do now.
	\begin{Lemma}
		Given $a,b \in \mathcal{O}_K$ with $b \neq 0$, there exists $q,r \in \mathcal{O}_K$ such that $a = bq + r$ with $|r| \leq \frac{\sqrt{15}}{4}|b|$.
	\end{Lemma}
	\begin{proof}
		Write $\frac{a}{b} = c_0+c_1\sqrt{d} \in \qq[\sqrt{d}]$. If $d = -1,-2$, pick $s,t$ to be the closest integers to $c_0$ and, $c_1$ respectively, so that $|c_0-s|\leq \frac{1}{2}$ and $|c_1-t|\leq \frac{1}{2}$. Then 
		\[
		\left|\frac{a}{b} - (s+t\sqrt{d}) \right|^2= \left|(c_0-s) + (c_1-t)\sqrt{d}\right|^2 \leq \frac{1}{4} + 2\frac{1}{4} = \frac{3}{4}.
		\]
		Thus 
		\[
		|a-(s+t\sqrt{d})b| \leq \frac{\sqrt{3}}{2}|b| \leq \frac{\sqrt{15}}{4}|b|.
		\]
		
		If $d = -3,-7,-11$, pick integers $s,t$ so that $\frac{t}{2}$ is as close to $c_1$ as possible, and $s$ has the same parity as $t$, and such that $\frac{s}{2}$ is as close to $c_0$ as possible. This gives that 
		\[
		\left| \frac{a}{b} - \frac{s+t\sqrt{d}}{2}\right|^2 = \left|c_0 - \frac{s}{2} + \left(c_1-\frac{t}{2}\right)\sqrt{d}\right|^2 \leq \frac{1}{4}+11\frac{1}{16} = \frac{15}{16}
		\] 
		so that 
		\[
		\left| a - b\frac{s+t\sqrt{d}}{2}\right| \leq \frac{\sqrt{15}}{4}|b|. \qedhere
		\]
	\end{proof}
	\begin{Cor}
		For $a_0,a_1\in \mathcal{O}_K$ if $|a_0|,|a_1| = O(|S|)$, then the sequence of successive remainders $a_{i+1} = a_{i-1}-q_ia_i$ terminates at $a_t = 0$ for some $t = O(\log|S|)$.
	\end{Cor}
	By the above Lemma, there exists algorithmic Euclidean division in $\mathcal{O}_K$ so that for any given $N,S,r \in \mathcal{O}_K$ we can define the sequences $a_k,b_k,c_k$ in exactly the same way as lines $1$-$5$ of Algorithm $\ref{ZiAlg}$ (with the bound $|r_t| \leq |a_t/\sqrt{2}|$ replaced with $|r_t|\leq \frac{\sqrt{15}}{4}|a_t|$.)
	\begin{Lemma}\label{ABLemmaR}
		For $k\geq 0$, $a_kb_{k+1}-a_{k+1}b_k = (-1)^kS$ and $|a_kb_{k+1}| \leq 16\left(1-\left(\frac{15}{16}\right)^{k} \right)|S|$.
	\end{Lemma}
	\begin{proof}
		Just as in the proof of Lemma \ref{ABLemma}, we have that $a_kb_{k+1}-a_{k+1}b_k = (-1)^kS$. The bound of $16\left(1-(15/16)^{k+1} \right)$ is also proved similarly. Supposing that $|a_{k-1}b_k| \leq 16\left(1-(15/16)^{k} \right)|S|$ then 
		\[
		|S| = |a_{k}b_{k+1}-a_{k+1}b_k| \geq |a_kb_{k+1}| - \frac{15}{16}|a_{k-1}b_{k}| 
		\]
		so that 
		\[
		|a_kb_{k+1}| \leq \left(1 + \frac{15}{16}\cdot16\left(1-(15/16)^{k}\right)\right)|S| = 16\left(1-(15/16)^{k+1} \right)|S|
		\]
		follows from induction. The base case $|a_0b_1| = |S| = 16(1-15/16)|S|$ holds by construction.
	\end{proof}
	\begin{Lemma}\label{XYTermGen}
		For nonzero $x,y,r,r',N \in \mathcal{O}_K$, if $(Sx+r)(Sy+r') = N$ with $|r|,|r'| < |S|\frac{\sqrt{15}}{4}$ and $|S|^3 > |N|$, then $|x|,|y| < 33|S|$ and $|xy| < 66|S|$.
	\end{Lemma}
	\begin{proof}
		This is proven exactly the same way as Lemma \ref{XYLemma} but with $\frac{1}{\sqrt{2}}$ replaced with $\sqrt{15}/4$. 
	\end{proof}
	Combining the above Lemmas gives 
	\begin{Lemma}
		Under the hypotheses of Lemma \ref{XYTermGen}, there exists $k$ with $$|a_kx+b_ky| < 530 |S|.$$
	\end{Lemma}
	\begin{proof}
		Since $|a_0| = |S|$, $|a_t| = 0$ and $|y| \leq 33|S|$ there exists a $k$ with 
		\[
		33|a_k| \leq |y| \leq 33|S|.
		\]
		This means that 
		\[
		|a_kx+b_ky| \leq |a_kx|+|b_ky| \leq \frac{1}{33}|xy| + 33|a_{k-1}b_k| < 530 |S|
		\]
		where the final inequality follows from the above two Lemmas.
	\end{proof}
	Thus an algorithm very similar to Algorithm $\ref{ZiAlg}$ can be used to find divisors in residue classes of $\mathcal{O}_K$ in polynomial time whenever $|S|^3 \geq |N|$.
	\subsection{Other Euclidean Number Rings}
	Can this new technique be extended to other Euclidean number rings besides the imaginary quadratics? The main idea behind both Lenstra's algorithm and our adaptation is the ability to restrict a congruence equation to a small number of explicit residue class. This has the effect of changing the congruence information into a small list of equalities that can be checked quickly. 
	
	However, if $K$ is not one of the fields $\qq(\sqrt{d})$ for $d = -1,-2,-3,-7,-11$, then either $\mathcal{O}_K$ is not Euclidean (in which case we cannot hope to define the sequences $a_i,b_i,c_i$) or else by Dirichlet's unit theorem there exists a unit $u \in \mathcal{O}_K$ of infinite rank. In particular, this means we can pick $u$ so that $|u| < 1$ (since $u$ is not a root of unity). This means that even if we had knowledge of say $x \equiv c \mod S$ and $|x| \leq B$ for some bound $B > |c|$, then if there exists any $\gamma \in \mathcal{O}_K$ such that $|c+\gamma S| < B$, then $|c+u^m\gamma S| < B$ for infinitely many values of $m$. This means that the congruence condition does not limit our search for solutions to some finite list (let alone polynomial!). 
	
	However, the question of finding all solutions to $(Sx+r)(Sy+r') = N$ for $S,r,N \in \mathcal{O}_k$ is well defined whenever $\mathcal{O}_K$ is a unique factorization domain. Can a different technique find solutions $x,y$ in this case? The ring $\zz[\sqrt{2}]$ seems like a good place to start if such a question has a positive answer.
	\section{The Algorithm in $\zz[x]$}\label{PolySec}
	The degree function in $\zz[x]$ shares many of the important properties of $|\cdot|$ in $K$. Namely, $\zz[x]$ is a Euclidean domain whose Euclidean function satisfies the (ultra-metric) triangle inequality and there are no units of infinite order. This means that we will be able to successfully prove analogues of Lemmas \ref{ABLemma} through \ref{SmallTerm}, and then successfully translate the bounded congruence constraints into a small number of equations that can be checked individually. What follows in this section are the proofs of the modified Lemmas as well as a description of the algorithm finding all solutions $(f,g) \in \zz[x]^2$ to $(Sf+r)(Sg+r') = N$ for a given $S,r,N$ with $\gcd(S,N) = \gcd(S,r) = 1$ and $3\deg(S)\geq \deg(N)$. We prove the algorithm's correctness and runtime and hence prove Theorem \ref{PolyTh}.
	As before, set $a_0 = S$, $b_0 = 0$, $c_0 = 0$ and $a_1 \equiv r'r^{-1} \mod S$, $b_1 = 1$, $c_1 \equiv \frac{N-rr'}{S}r^{-1} \mod S$, and define $a_{k+1},b_{k+1},c_{k+1}$ via the usual recurrence.
	
	\begin{Lemma}\label{EuclidDivZx}(Euclidean Division in $\qq(x)$)
		Given polynomials $A,B \in \qq[x]$, there exists unique $q,r \in \qq[x]$ such that $A = qB+r$ with $\deg(r) < \deg(B)$. Further, if $A,B \in \zz[x]$ and the coefficient of the leading term of $B$ divides the content of $A$ (or the gcd of it's coefficients) then $q,r \in \zz[x]$. In particular, this is the case when $B$ is a monic polynomial.
	\end{Lemma}
	\begin{proof}
		This is simply long division of polynomials.
	\end{proof}
	\begin{Lemma}
		For all $k$ we have $a_kb_{k+1}-a_{k+1}b_k = (-1)^k S$ and $\deg(a_kb_{k+1}) = \deg(S)$. 
	\end{Lemma}
	\begin{proof}
		The proof of the first fact is exactly the same as in Lemmas $\ref{ABLemma}$ and $\ref{ABLemmaR}$. 
		
		For $k = 0$ we have $\deg(a_0b_1) = \deg(S)$ by construction, and for $k = 1$
		\[
		\deg(S) = \deg(a_1b_2-a_2b_1) = \deg(a_1b_2 - a_2) = \deg(a_1b_2)
		\] where the final equality follows from the fact that $\deg(a_2) < \deg(a_1)\leq \deg(a_1b_2)$. Finally, for $k\geq 1$ we have that 
		\[
		\deg(S) = \deg(a_{k+1}b_{k+2}-a_{k+2}b_{k+1}) = \deg(a_{k+1}b_{k+2}).
		\]
		Again, the final equality holds since $\deg(a_{k+2}) < \deg(a_{k})$, so that the inequality $\deg(a_{k+2}b_{k+1}) < \deg(a_kb_{k+1}) = \deg(S)$ holds. \qedhere
	\end{proof}
	\begin{Lemma}\label{XYtermPoly}
		For $f,g,r,r',N \in \zz[x]$ nonzero, if $$(Sf+r)(Sg+r') = N$$ with $\deg(r),\deg(r') < \deg(S)$, and $3\deg(S) \geq \deg(N)$, then $\deg(f)+\deg(g)\leq \deg(S)$. In particular $\deg(f),\deg(g) \leq \deg(S)$.
	\end{Lemma}
	\begin{proof}
		Since $N = (Sf+r)(Sg+r') = S^2fg + S(fr'+gr) + rr'$ we have that 
		\[
		3\deg(S) \geq \deg(N) \geq \deg(S^2fg) = 2\deg(S) + \deg(f) + \deg(g). \qedhere
		\]
	\end{proof}
	\begin{Lemma}\label{SmallTermPoly}
		Under the hypotheses of Lemma \ref{XYtermPoly}, there exists $k$ with $$\deg(a_kf+b_kg) \leq \deg(S).$$
	\end{Lemma}
	\begin{proof}
		Pick $k$ so that $\deg(a_{k+1}) \leq \deg(g) \leq \deg(a_k)$. This is possible since $\deg(a_0) = \deg(S) \geq \deg(g)$ and $\deg(a_t) = 0 \leq \deg(g)$. Then 
		\[
		\deg(a_{k+1}f+b_{k+1}g) \leq \max(\deg(fg),\deg(a_kb_{k+1})\leq \deg(S)\qedhere
		\]
	\end{proof}
	\begin{Alg}\label{PolyAlg}
		\mbox{}
		
		\noindent	\textit{Inputs:}
		\begin{itemize}
			\item[-] Polynomials $S,N,r$ with $\deg(S)^3 \geq \deg(N)$, and $\gcd(N,S) = \gcd(S,r) = 1$, a list $L$ of the divisors of $\ell(N)/\left(\ell(S)^2\right)$ in $\zz$, where $\ell(p)$ is the leading coefficient of $p \in \zz[x]$. Note that if $\left(\ell(S)^2\right) \nmid \ell(N)$ the list $L$ is empty.
		\end{itemize}
		\textit{Outputs:}
		\begin{itemize}
			\item[-] A complete list of polynomials $d \in \zz[x]$ such that $d$ divides $N$, and $d \equiv r \mod S$.
		\end{itemize}
		\begin{algorithmic}[1]
			\STATE Set $r = r \mod S$, and initialize $a_0 = S, a_1 = r'r^{-1} \mod S$, $b_0 = 0, b_1 = 1$, $c_0 = 0, c_1 = \frac{N-rr'}{S}r^{-1} \mod S$. Finally, set $r' = Nr^{-1} \mod S$ and set $t = 1$.
			\STATE If $r | N$, add $r$ to the list of outputs. If $r'|N$, add $N/r'$ to the list of outputs.
			\WHILE{$a_t \neq 0$}
			\STATE Calculate $q_t,r_t$ via Lemma \ref{EuclidDivZx} so that $a_{t-1} = a_{t}q_{t} + r_t$ and $\deg(r_t) < \deg(a_{t})$.
			\STATE Set $t = t+1$, and $a_t = r_{t-1}$. In addition, set $b_t = b_{t-2} - q_{t-1} b_{t-1}$ and set $c_t = c_{t-2}-q_{t-1}c_{t-1} \mod S$.
			\ENDWHILE 
			\FOR{$i=1,\dots, t$}
			\FOR{$d \in L$}
			\STATE Solve the system $a_if+b_ig = c_i$ and $(Sf+r)(Sg+r') = N$ for $f,g$. If $f,g$ are elements of $\zz[x]$ add $Sf+r$ to the list of outputs.
			\STATE Solve the system $a_if+b_ig = c_i+\left(\ell(a_i)d + \ell(b_i)\frac{\ell(N)}{\ell(S)^2d}\right)/\ell(S)S$ and $(Sf+r)(Sg+r') = N$ for $f,g$. If $f,g$ are elements of $\zz[x]$ add $Sf+r$ to the list of outputs.
			\ENDFOR
			\ENDFOR
			\STATE Return the list of outputs.
		\end{algorithmic}
	\end{Alg}
	\begin{Prop}\label{PolyAlgTh}
		Algorithm $\left(\ref{PolyAlg}\right)$ is correct, and runs in $O(|L|\deg(N))$ elementary operations on elements in $\zz[x]$. This proves Theorem $\ref{PolyTh}$
	\end{Prop}
	\begin{proof}
		We first show correctness. To do this, we must show that if $(Sf+r)(Sg+r') = N$ is a solution that the algorithm recovers it. If either $f$ or $g$ is equal to $0$, then the correct divisor is discovered in line $2$. We now assume that neither $f$ nor $g$ is the zero polynomial.
		
		By Lemma \ref{SmallTermPoly}, there exists a value $0\leq *\leq t$ such that $\deg(a_*f+b_*g) \leq \deg(S)$ and $a_*f+b_*g \equiv c_* \mod S$. This means that $a_*f+b_*g = c_* + pS$ for some $p\in\qq[x]$ where either $p = 0$ or $\deg(p) = 0$ (recall that the terms $a_*,b_*$ may be in $\qq[x]$ and not $\zz[x]$). If $p = 0$, the algorithm will recover $f$ and $g$ in line $8$. 
		
		We now consider the case $p \neq 0$. Since $(Sf+r)(Sg+r') = N$, we have that $\ell(N) = \ell(S)^2\ell(f)\ell(g)$ meaning that $\ell(f) = d_*$ for some $d_* \in L$ and $\ell(g) = \frac{\ell(N)}{\ell(S)^2d_*}$. By comparing the leading terms of both sides of the equation $a_*f+b_*g = c_*+pS$ we obtain 
		\[
		p = \left(\ell(a_i)d + \ell(b_i)\frac{\ell(N)}{\ell(S)^2d}\right)/\ell(S)
		\]
		so that the algorithm will recover $f,g$ in line $9$. 
		
		We now examine the run time. Both steps $1$ and $2$ are trivial. Each iteration of the loop entered in step $3$ requires one polynomial long division, one integer division, and two polynomial multiplications and subtractions. Since $\deg(a_{k+1}) < \deg(a_k)$, this loop is entered at most $O(\deg(S))$ times. Similarly $t \leq \deg(S)$.
		
		The nested loop entered in lines $6$ and $7$ requires at most $O(|L|t)$ additions, multiplications and square root extractions of polynomials in $\zz[x]$. This completes the proof of Proposition $\ref{PolyAlgTh}$.
	\end{proof}
	\begin{Rem}
		It is perhaps slightly unfair to consider square root extraction of a polynomial as an elementary operation. If $p \in \zz[x]$ we describe briefly how it can be determined if $p$ is a square in $\zz[x]$, and if so, return the square-root in $O(\deg(p))$ elementary operations on integers (including $1$ integer square root extraction). We write $p = p_0+p_1x + \cdots +p_{m}x^{m}$. If $p_0$ is not a square in $\zz[x]$, or if $m$ is odd, then $p$ is not a square in $\zz[x]$. Otherwise, write $p_0 = a_0^2$ and $n = m/2$, and define the recurrence relation 
		\[
		a_{i+1} = \frac{p_{i+1}-\sum_{k=1}^{i}a_{k}a_{n+1-k}}{2a_0}.
		\]
		If $a_i \in \zz$ for $0\leq i \leq m$ and $a_i = 0$ for $n < i \leq m$, then set $A = a_0 + a_1x+\cdots + a_nx^n$. Then $A^2 = p$ if and only if $p$ is a square in $\zz[x]$. This works because the recurrence relation defined above gives the coefficients of the power-series expansion of $\sqrt{p}$ as an element of $\cc[[x]]$.
		
		We also remark that it is easy to adapt Algorithm $\ref{PolyAlg}$ to the case $\mathcal{O}_K[x]$ for any imaginary quadratic field $K = \qq(\sqrt{d})$ of class number $1$ (not just the Euclidean domains). In general for other number rings we will either not have unique factorization (i.e. the class number is larger than $1$) or the rank of the unit group will be at least $1$ (so that $L$ is not a finite list).
	\end{Rem}
	\section{A Note on Upper and Lower Bounds}\label{BoundsSec}
	In this section we discuss the known upper and lower bounds on the number of divisors an integer can have in a given residue class. 
	We do this in two subsections dedicated to upper and lower bounds respectively. 
	In the first section we summarize techniques due to Lenstra, and Coppersmith, Howgrave-Graham, and Nagaraj. We also briefly describe how their methods can be adapted to the number fields in Theorem \ref{ImagQuadTh}.
	In the second section we describe the current state of the lower bound problem and summarize work due to Henri Cohen and provide some previously unknown examples (and families of examples). 
	\subsection{Upper Bounds}
	In Lenstra's original paper \cite{DivsInRes}, he attempts to answer the problem of how many outputs his algorithm can produce, and proved that for any integers $N,S,r$ with $\gcd(N,S) = \gcd(S,r) = 1$ that the number of divisors of $N$ congruent to $r$ mod $S$ is less than a constant $c(\alpha)$ depending only on $\alpha$ whenever $S^{\alpha} > N$ and $\alpha > 1/4$.
	Lenstra did this by proving the following general fact about weight functions of finite sets.
	
	\begin{Th}[H.W. Lenstra]\label{LenstraCombo}
		Let $V$ be a finite set, $\omega$ a weight function on $V$ (a finite measure on $\mathcal{P}(V)$) and $\alpha \in \mathbb{R},\alpha > 1/4$. Further let $\mathcal{D}$ be a system of subsets of $V$ such that 
		\[
		\max\left\{\omega(D_1-D_2),\omega(D_2-D_1) \right\} \geq \alpha \omega(V)
		\]
		for all $D_1,D_2\in \mathcal{D}$. Then $|\mathcal{D}| \leq c(\alpha)$, where $c(\alpha)$ is a constant that depends only on $\alpha$. In addition, $c(1/3) = 12$.
	\end{Th}
	To apply Theorem \ref{LenstraCombo} to the case of divisors in residue classes, for each positive integer $N$, Lenstra defines 
	\[
	V(N) := \{p^t: p^t | N, p\text{ prime }, t \geq 1\}
	\]
	with weight function defined via $\omega(\{p^t\}) = \log(p)$. Then for the system of subsets of $V(N)$
	\[
	\mathcal{D} := \{V(d): d|N, d\equiv r \mod S\}
	\]
	if $d_1 = Sx+r$ and $d_2 = Sy+r$ are positive divisors of $N$ with $x > y$ we have 
	\begin{align*}
	\omega\left(V(d_1)-V(d_2)\right) = \log\left(\frac{d_1}{\gcd(d_1,d_2)}\right) &\geq \log\left(\frac{Sx+r}{\gcd(Sx+r,Sy+r)}\right)\\
	& \geq \log\left(\frac{Sx+r}{x-y}\right) \geq \log(S) \geq \alpha \omega(V).
	\end{align*}
	The second inequality follows from the fact that 
	\[
	\gcd(Sx+r,Sy+r) = \gcd(S(x-y),Sy+r) = \gcd(x-y,Sy+r)
	\]
	since $\gcd(S,r) = 1$.
	While a theoretical breakthrough, the upper bound $c(\alpha)$ is not given constructively, and a large portion of \cite{DivsInRes} is dedicated to proving that $c(1/3) = 12$ and that $c(\alpha) = O\left((\alpha-1/4)^{-2}\right)$.
	
	Twentey-three years later, Coppersmith, Howgrave-Grahm, and Nagaraj were able to modify the Coppersmith method of \cite{CopMethod} to give a constructive upper bound, but importantly their model only applies to integers $n$ sufficiently large (depending on $\alpha$), showing that $c(\alpha) = O((\alpha-1/4)^{-3/2})$.
	\begin{Th}[Coppersmith, Howgrave-Grahm, Nagaraj]
		Given $\alpha > 1/4$ there is an integer $n_0 > 0$ such that for all integers
		$n>n_0$ and $s>n^{\alpha}$ and $0 <r<s<n$ with $\gcd(r, s) = \gcd(s, n)=1$, the number
		of divisors $d$ of $n$ with $d \equiv r \mod s$ is bounded above by
		\[
		2+\frac{2\pi\alpha}{(\alpha-1/4)^{3/2}} + \frac{4\alpha}{\alpha-1/4}.
		\]
	\end{Th}
	While theoretically an improvement, the authors note that the upper bounds they obtain are often more than double the select values calculated by Lenstra. See table $1$ in \cite{DivsInResCon}. The authors believed that this discrepancy was due to the fact that their method counts both positive and negative integers $d$ such that $d \equiv r \mod s$. Said another way, it counts positive divisors of $n$ congruent to $\pm r \mod s$. 
	
	It is an interesting question to determine if Lenstra's combinatorial model can be used to count all divisors (positive and negative) in a single residue class, or more generally, in the number rings $\mathcal{O}_k$ of Theorem \ref{ImagQuadTh}. We offer a quick proposition as a partial solution.
	\begin{Prop}\label{PMDivs}
		For all $\epsilon >0$, there exists $N_\epsilon$ so that for all $N > N_\epsilon$ if $S,r$ are any positive integers with $S^{\alpha + \epsilon} > N$ and $\gcd(S,N) = \gcd(S,r) = 1$ then $N$ has at most $c(\alpha)$ positive and negative divisors congruent to $r \mod S$ where $c(\alpha)$ is a constant that depends only on $\alpha$.
	\end{Prop}
	\begin{proof}
		Fix $\epsilon >0$, set $N_\epsilon = 2^{1/\epsilon}$ and suppose that $S,r,N$ satisfy the hypotheses of Proposition \ref{PMDivs}. As before, for positive integers $d$ let \[
		V(d) := \{p^t: p^t| N, p\text{ prime}, t\geq 1\}
		\]
		and for negative integers $d$, define 
		\[
		V(d) = \{-1\} \cup V(|d|).
		\]
		Using the notation of Theorem \ref{LenstraCombo}, let $V = V(-N)$, and \[
		\mathcal{D} := \{V(d): d\in\zz, d|N, d\equiv r \mod S\}.
		\]
		We also define $\omega(\{p^t\}) = \log(p)$ and $\omega\{-1\} = 0$ and now show that the hypotheses of Theorem \ref{LenstraCombo} are met by this system.

		Let $d_1 = Sx+r$ and $d_2 = Sy+r$ and without loss of generality we assume that $|x| \geq |y|$ and that if equality holds, then $x > 0$, and by the division algorithm, we may also assume that $|r| \leq S/2$. We calculate directly that 
		\begin{align*}
		\omega\left(V(d_1)-V(d_2)\right) &= \log\left( \frac{|d_1|}{\gcd(d_1,d_2)}\right) \geq \log\left(\frac{S|x|-S/2}{2|x|} \right) \\
		&\geq\log(S) +\log\left(\frac{|x|-1/2}{|x|}\right)\geq \log(S) -\log(2).
		\end{align*}
		Since $S^{\alpha+\epsilon} \geq N$ we have that 
		\[
		\log(S)-\log(2) \geq (\alpha+\epsilon)\log(N)-\log(2) \geq \alpha\log(N) = \omega(V).
		\]
		The penultimate equality follows from the fact that 
		\[
		\epsilon\log(N) \geq \epsilon\log(N_\epsilon) \geq \log(2)
		\]
		by construction. Thus by Theorem \ref{LenstraCombo} we have that $|\mathcal{D}| \leq c(\alpha)$ as desired.
	\end{proof}
	Before continuing onto lower bounds, we offer two quick comments. First, Proposition \ref{PMDivs} shows that the rough doubling of the bounds given by Coppersmith et al. is not due to counting both positive and negative divisors since a minor alteration to Lenstra's model can account for both positive and negative divisors. 
	Secondly, Proposition \ref{PMDivs} can be easily modified to create an upper bound on the number of solutions to $(Sx+r)(Sy+r) = N$ in $\mathcal{O}_K$ for the number fields of Theorem \ref{ImagQuadTh}, although $N_\epsilon$ grows to $\left(32+8\sqrt{15}\right)^{1/\epsilon}$ for field $K = \qq(\sqrt{-11})$.
	\subsection{Lower Bounds}
	In addition to upper bounds, some work has gone into establishing, via example, lower bounds for the maximum number of divisors a rational integer $N$ can have in a single residue class $r \mod S$, with $(N,S) = (S,r) = 1$, and $S > N^{\alpha}$ for a fixed $\alpha$. For example, the case of $\alpha = \frac{1}{2}$ is completely known, as the Lenstra upper bound $c(1/2) = 2$ is achieved by infinitely many tuples $(N,S,r)$. In fact, for $n\geq 3$ we have that for $(N,S,r) = (N,N-1,1)$ that $N$ has two divisors (namely $1$ and $N$) that are both congruent to $1 \mod N-1$, and $N-1 > N^{1/2}$.
	
	The case is much less clear when $\alpha$ is closer to $1/4$, the threshold of both the Lenstra and Coppersmith et al. methods with significant work only invested in the case $\alpha = 1/3$. 
	In 1982 Henri Cohen showed in \cite{COHEN6Divs} that for any $\ell \geq 3$ that for 
	\begin{align*}
	N(\ell) &= (2\ell+1)(\ell^2+1)(\ell^2+\ell+1)(2\ell^2-\ell+1)(2\ell^2+\ell+1)\\
	S(\ell) &= 2\ell^3+\ell^2+2\ell
	\end{align*}
	that $S(\ell)^3 > N(\ell)$ and that $N(\ell)$ has exactly $6$ divisors congruent to $1\mod S(\ell)$. Note that since $S(\ell)^3$ and $N(\ell)$ have the same leading term, that 
	\[
	\alpha(\ell) := \log_{N(\ell)}(S(\ell))
	\]
	rapidly approaches (from above) the limit $1/3$.
	
	For $\alpha \geq 1/3$ It's currently an open problem as to whether $6$ is the best possible, although there does seem to be some hope that $7$ (or more) divisors may be obtainable. 
	For example, if $ N = 104254876089000$, and $S = 105787$, then $N$ has $6$ divisors congruent to $1$ mod $S$, and $\alpha = \log_N(S) \approx 0.3584$ which is significantly larger than $1/3$. This might suggest that there is enough room to add an additional divisor and maintain the threshold of $\alpha > 1/3$.
	
	Additionally, for
	\begin{align*}
	N(x) &= (x + 2)  (x + 1)^2 (x^2 + x + 1) (x^2 + x + 2) (x^2 + 2x + 2)\\
	S(x) &= x^3 + 3x^2 + 4x + 3
	\end{align*}
	the integers $N(x)$ have $7$ divisors congruent to $1 \mod S(x)$ among both positive and negative divisors for any $x \geq 2$.
	These examples of $(N,S,r)$ are new discoveries and do not belong to Cohen's infinite family.
	\section{Summary and Future Work}\label{conc}
	In this work we generalized Lenstra's classical algorithm for finding divisors in residue classes to several imaginary quadratic rings as well as the polynomial ring $\zz[x]$. 
	This required a new approach as Lenstra's original method critically relied on the total ordering of the integers and on a critical sign change that do not generalize to higher rings of integers. 
	We proved that our generalization runs in polynomial time in $\log(|N|)$ and implemented our method providing numerical results.
	
	In concluding, we offer three questions of future research:
	\begin{enumerate}
		\item Can a generalization of our new method be applied to other Euclidean rings with a non-finite unit group? For example in the ring $\zz[\sqrt{2}]$ for any element $\alpha \in \zz[\sqrt{2}]$ there is a unique (up to sign) associate of $\alpha = u\alpha'$ for some unit $u$ and $\alpha'\in\zz[\sqrt{2}]$ such that $1\leq |\alpha'|<1+\sqrt{2}$. 
		Can this allow us to prove a version of Lemma \ref{SmallTerm} for the ring $\zz[\sqrt{2}]$?
		\item Because the discrepancy between the Lenstra and Coppersmith et al. upper bounds for the number of divisors in a single residue class is not due to counting both positive and negative divisors, why does the Lenstra method produce such significantly smaller upper bounds?
		\item Can new examples of integers with many divisors in a single residue class? Is Lenstra's upper bound sharp for values of $\alpha < 1/2$? Specifically, can an integers $N,S,r$ be found with $\gcd(N,S) = \gcd(S,r) = 1$ and $S^3 > N$ so that $N$ has at least $7$ divisors congruent to $r \mod S$?
	\end{enumerate}
	\bibliographystyle{plain}
	\bibliography{Divisors_in_Residue_Classes_Revisited}
\end{document}